\documentclass[12pt,reqno]{amsart}
\usepackage{amsmath, amscd}
\usepackage{amsfonts}
\usepackage{version}
\usepackage{amssymb}
\usepackage{amssymb,amsmath}
\usepackage{ifthen}
\usepackage{epsfig}
\usepackage{graphicx}
\usepackage{caption}
\usepackage{subcaption}

\newcommand{\Gauss}{{\null_2F_1}}

\newcommand{\al}{\alpha}

\newcommand{\C}{\mathbb C}

\newcommand{\D}{\mathbb D}

\newcommand{\Ao}{\mathfrak A}

\newcommand{\Ff}{\mathfrak F}
\newcommand{\Gf}{\mathfrak G}

\newcommand  {\Id} {\mathop{\rm Id}\nolimits}

\renewcommand{\Re}{\mathop{\rm Re}\nolimits}

\newtheorem{theorem}{Theorem}[section]
\newtheorem{lemma}[theorem]{Lemma}

\newtheorem{corollary}[theorem]{Corollary}
\newtheorem{question}{Question}
\newtheorem{conjecture}{Conjecture}

\newtheorem{definition}[theorem]{Definition}

\numberwithin{equation}{section}

\numberwithin{equation}{section}

\begin{document}
\title{On the hypergeometric function and families of holomorphic functions}

\author[M. Elin]{Mark Elin}
\address{M. Elin: Department of Mathematics, Braude College of Engeneering, Karmiel 21982, Israel.}
\email{mark\_elin@braude.ac.il}

\author[F. Jacobzon]{Fiana Jacobzon}
\address{F. Jacobzon: Department of Mathematics, Braude College of Engeneering, Karmiel 21982, Israel.}
\email{fiana@braude.ac.il}

\date{\today}
\subjclass[2020]{Primary 30C55; Secondary 33C05}
\keywords{hypergeometric function; filtration; lattice; quasi-extremum}

\begin{abstract}
In this work, we examine one two-parameter family of sets consisting of  functions holomorphic in the unit disk, previously investigated by several mathematicians. We focus on the set-theoretic properties of this family, identify the general form of filtrations within it, and discover that it is not a lattice. This insight motivates us to introduce a refined concept of quasi-infima and quasi-suprema, and to establish their complete description.

Unexpectedly, some new properties of the Gau\ss\  hypergeometric function play a crucial role in our investigation. 
\end{abstract}

\maketitle

\section{Introduction}\label{sect-introd}

The paper explores  sets $\Ao_s^t$ of functions that are holomorphic in the open unit disk $\D$, normalized by $f(0)=f'(0)-1=0$  and satisfy the inequality
\[
\Re \left[(s-1) \frac{f(z)}{z} + f'(z) \right]\geq st, \  z\in \D\setminus \{0\}, 
\]
where $s>0$ and $0\ge t<1$. In addition to intrinsic interest, these sets appeared  in the investigation of extreme points of classes of univalent functions in~\cite{Hal}, in a relation to certain integral transforms, see~\cite{K-R}, as well as in the study of infinitesimal generators of semigroups in~\cite{E-J-survey}. For more results on different families of holomorphic functions the reader can consult the book \cite{GAW}.  Here we are interested in the set-theoretic structure of the family $\Ao:= \left\{ \Ao_s^t\right\}$.   \vspace{2mm}

It appears that to investigate certain set-theoretic properties, a prerequisite understanding of Gau\ss\  hypergeometric functions is necessary.  In this connection, it should be noted that in recent decades many authors have studied geometric properties of hypergeometric functions (see, for example,  \cite{BPV02, SW17, W22}). New results regarding sums of products and ratio of hypergeometric functions were established in \cite{D-K-21, K-K-21}. In \cite{OP},  the zero-balanced  hypergeometric function $\Gauss(1,s;s+1;z)$ was applied to establishing new conditions for univalence and starlikeness of certain transforms.
\vspace{2mm}
 
  Section~\ref{sect-some_functions} considers a zero-balanced  hypergeometric function $\Gauss(1,s;s+1;z)$.  We discover its subtle  characteristics as a function of  $s$. In the subsequent sections, we elaborate on an approach that capitalizes on the dependence of the hypergeometric function $\Gauss(1,s;s+1;z)$ on its parameter.

In Section~\ref{sect-2-param}, we concentrate on the two-parameter family  $\Ao$ which is  the main object of the study in this paper. Conditions that entail/exclude the inclusion of two elements of this family into one another are derived.  \vspace{2mm}

The results on the inclusion relation are applied in Section~\ref{sect-filtra} to answer our main questions. The first one is

$\bullet$ {\it How to characterize all filtrations included in this family?   } Recall that a one-parameter family of sets $\left\{\Ff_t\right\}$ is a filtration (see, for example, \cite{BCDES, E-J-survey, ESS}) if it is ordered, more precisely, $\Ff_s\subset\Ff_t$ whenever $s<t$.

This problem is partially addressed in \cite{E-J-survey}. In Theorem~\ref{thm-filtr} we give the complete answer. 

\vspace{2mm}
Another question is 

$\bullet$ {\it Is the whole family a lattice?   } Recall that a partially ordered family $\Gf=\left\{\Gf_\al\right\}$ endowed with the relation $\subset$  is 
 lattice if each pair of elements has the unique supremum and the unique infimum.

By definition, the supremum of the pair $\Gf_1,\Gf_2\in\Gf$  (if it exists) is the element of $\Gf$  denoted by  $\sup(\Gf_1,\Gf_2)$   such that  $\Gf_1 \cup \Gf_2 \subset \sup(\Gf_1,\Gf_2 )$ and if $\Gf_1\cup\Gf_2 \subset \Gf_*$ for some $\Gf_*\in\Gf$, then $\sup(\Gf_1,\Gf_2 )\subset \Gf_*$. Analogously, the infimum is the element $\inf(\Gf_1,\Gf_2 )$ such that $\inf(\Gf_1,\Gf_2)\subset \Gf_1\cap \Gf_2$ and the inclusion $\Gf_*\subset \Gf_1\cap\Gf_2$ implies $\Gf_*\subset\inf(\Gf_1,\Gf_2)$.

\vspace{2mm}
Definition~\ref{def-quasi} introduces refined concepts: sets of quasi-infima and quasi-suprema. We give the complete description of quasi-extrema for each pair of elements of $\Ao$ in Theorem~\ref{thm-supre}.

Furthermore, the observation below shows that if a pair $\Gf_1, \Gf_2 \in\Gf $ has a supremum, then the quasi-supremum coincides with the supremum and so is unique. Since, according to our results, it is not the case that for every pair of elements of $\Ao$ there is a unque quasi-supremum, we conclude: 
 
 \vspace{1mm}
{\it  The family $\Ao= \left\{ \Ao_s^t\right\}$ is not a lattice.}

\vspace{2mm}
In the last Section~\ref{sect-conclu}, we pose several questions for a forthcoming investigation.

\bigskip

\section{Some new properties of the hypergeometric function}\label{sect-some_functions}

\setcounter{equation}{0}

To prove the main result of this section we need two auxiliary lemmata. 

\begin{lemma}\label{lem-uniq-root}
Let $\psi_1$ and $\psi_2$ be continuous functions defined for $x>0$ by the formulas
 \begin{equation*}
\psi_1(x):=  \frac{2(1+x) }{ x^2}\,  \log\left(1+\frac{x^2}{4(1+x)}\right), \quad \psi_2(x):= \frac{2+x +(1+x)\log(1+x)}{(2+x)^2}
\end{equation*}
and $\psi_1(0)=\psi_2(0)=\frac12$. Then the equation $\psi_1(x)=\psi_2(x)$ has a unique solution in $(0,\infty)$.
\end{lemma}

The proof of this lemma is very technical and long. For this reason, we present it in Appendix at the end of the paper.

The next assertion is a simple consequence of the theorem on integral average.
\begin{lemma}\label{lem-integral}
  Let $-\infty\le a<b\le\infty$ and functions $\phi,\psi\in C(a,b)$ satisfy
  \begin{itemize}
    \item [(i)] $\phi$ is bounded, positive and decreasing;
    \item [(ii)] there is $t_0\in(a,b)$ such that $\psi(t)<0$ as $t\in(a,t_0)$ and  $\psi(t)>0$ as $t\in(t_0,b)$;
    \item [(iii)] the improper integral $\displaystyle\int_a^b \psi(t)dt$ equals zero.
  \end{itemize}
  Then $\int_a^b \phi(t) \psi(t)dt <0$.
\end{lemma}

\begin{proof}
Conditions (ii) and (iii) imply that $\int_{t_0}^b \psi(t)dt=- \int_a^{t_0} \psi(t)dt>0. $  Therefore for any $t_1\in(a,t_0)$ there is a unique $t_2\in(t_0,b)$ such that
\[
0< \int_{t_0}^{t_2} \psi(t)dt=- \int_{t_1}^{t_0} \psi(t)dt =: A(t_1)
\]
and $t_2\to b^-$ as $t_1\to a^+$. By the integral average theorem, there are points $t^*\in(t_1,t_0)$ and $t^{**}\in(t_0,t_2)$ such that
\begin{eqnarray*}
 \int_{t_1}^{t_0} \phi(t)\psi(t) dt  &= & \phi(t^*) \int_{t_1}^{t_0} \psi(t) dt =  - \phi(t^*)  A(t_1),             \\
 \int_{t_0}^{t_2} \phi(t) \psi(t) dt  &=& \phi(t^{**})\int_{t_0}^{t_2} \psi(t) dt =  \phi(t^{**})  A(t_1).
\end{eqnarray*}

Thus
\begin{eqnarray*}
  \int_a^b\phi(t)\psi(t)dt &=&  \lim_{t_1\to a^+} \left[ \int_{t_1}^{t_0} \phi(t) \psi(t) dt +  \int_{t_0}^{t_2} \phi(t) \psi(t) dt \right]   \\
   &=&   \lim_{t_1\to a^+} \left[  - \phi(t^*)  A(t_1) +  \phi(t^{**})  A(t_1) \right]   \\
   &=&  \lim_{t_1\to a^+} \left[  - \phi(t^*)   +  \phi(t^{**})  \right] A(t_1) <0
\end{eqnarray*}
because $t^*<t_0<t^{**}$ and thanks to condition (i).
 \end{proof}
 
 Choosing in this lemma $\phi(t)=e^{-st}$, we conclude the following:
\begin{corollary}\label{corr-laplace}
  Let function $\psi\in C(0,\infty)$, $\psi(t)<0$ as $t\in(0,t_0)$  for some   $t_0\in(0,\infty)$, $\psi(t)>0$ as $t\in(t_0,\infty)$, and  $\displaystyle\int_0^\infty \psi(t)dt=0$.
  Then the Laplace transform $\mathcal{L}[\psi](s)$ is negative in $s>0$. 
 \end{corollary}

\hspace{2mm}



We now turn to the Gau{\ss} hypergeometric function $\Gauss(a,b;c;\cdot)$. Here $a,b,c\in\C$ are parameters that satisfy $0<\Re b<\Re c$. Recall that this function is defined for $z\in\D$ by  
\begin{equation}\label{gauss}
\Gauss(a,b;c;z)= 1+ \sum_{n=1}^\infty\frac{(a)_n(b)_n}{(c)_nn!}z^n=\frac{\Gamma(c)}{\Gamma(b)\Gamma(c-b)} \int_0^1\frac{x^{b-1}(1-x)^{c-b-1}}{(1-zx)^a}dx,
\end{equation}
where $(\alpha)_n=\frac{\Gamma(\al+n)}{\Gamma(\al)}=\al\cdot(\al+1)\cdot\ldots\cdot(\al+n-1)$ is the Pochhammer symbol.
For geometric properties of $\Gauss(a,b;c;z)$, we refer to the useful papers \cite{BPV02, SW17, W22} and the references therein. If $c=a+b$, the hypergeometric function $\Gauss(a,b;a+b;z)$  is called {\it zero-balanced}.

We now consider the following functions:

\begin{equation}\label{xi_0}
\xi_0(s):=2 \Gauss(1,s;s+1;-1)-1 =\int_0^1\frac{1-x}{1+x}\,sx^{s-1} dx
\end{equation}
and
\begin{equation}\label{all--xi}
\xi_1(s):=\frac{1-\xi_0(s)}{2s}\, , \quad \xi_2(s):=2s\xi_0(s) ,  \quad \xi_3(s):=   \frac{1-\xi_0(s)}{2s\xi_0(s)}\,, \quad s>0     .
\end{equation}

\begin{theorem}\label{lem-kappa}
The functions $\xi_0,\,\xi_1,\,\xi_2$ and $\xi_3$ are continuous on $(0,\infty).$ Moreover, 
\begin{itemize}
  \item [(i)] function $\xi_0$ is decreasing and maps $(0,\infty)$ onto $(0,1)$ and such that the function $s\mapsto s^2\xi_0'(s)$ is decreasing;
  \item [(ii)] function $\xi_1$ is decreasing  and maps $(0,\infty)$ onto $(0,\ln2)$;
  \item [(iii)]function $\xi_2$ is increasing and maps $(0,\infty)$ onto $(0, 1)$;
\item [(iv)] function $\xi_3$ is increasing and maps $(0,\infty)$ onto $(\ln2,1)$.
\end{itemize}
Thus, since these functions are monotone, they can be extended to $[0,\infty)$ and even be defined by continuity at $\infty$.
\end{theorem}

\begin{proof}
 Since $\displaystyle \xi_0'(s)= \int_0^1\frac{1-x}{1+x}\cdot\frac{\partial}{\partial x}\left(x^s\ln x\right) dx=\int_0^1\frac{2x^s \ln x}{(1+x)^2}\, dx<0$,  function $\xi_0$ is decreasing.   In addition, ${\displaystyle (s^2\xi_0'(s))'=-2\int\limits_0^1 \frac{ 1-x } {(1+ x)^3}x^s\ln^2 x\,dx<0}$, so, statement (i) follows.
 
  Further, note that $\displaystyle  \xi_1(s)= \int_0^1 \frac{x^s}{1+x}\,dx,$  which implies statement (ii).

As for function $\xi_2$, fix arbitrary $s_2>s_1>0$. According to Cauchy's mean value theorem applied to the functions $\xi_0(s),1/s\in C[s_1,s_2]$, there is $\tilde{s}\in(s_1,s_2)$ such that 
$$\displaystyle\frac{\xi_0'(\tilde s)}{-1/\tilde s^2}= \frac {\xi_0(s_2) -\xi_0(s_1)} {1/s_2-1/s_1}.$$ 
Since the function  $s^2\xi_0'(s)$ is decreasing, $ s_1^2\xi'_0(s_1) > \tilde s^2\xi_0'(\tilde s) =-  \frac {\xi_0(s_2) -\xi_0(s_1)} {1/s_2-1/s_1}$. Letting ${s_2\to\infty}$, we conclude that  $ s_1 \xi'_0(s_1)  > -  \xi_0(s_1)$. Because the point $s_1$ is arbitrary, one has $\frac{\xi_0'(s)}{\xi_0(s)} +\frac1s > 0$, or, which is the same, $\left(\log\xi_2(s) \right)'  > 0$. Thus statement (iii) is proved. 

To prove statement (iv), one has show $\xi_3'(s) > 0$. This inequality is equivalent to
\begin{equation}\label{g}
 g(s) <0,\quad\mbox{where}\quad g(s):=(1- \xi_0(s) )\xi_0(s) +s\xi_0'(s)       
\end{equation}

Return to the integral in \eqref{xi_0} defining the function $\xi_0$ and substitute there $x=e^{-t}$:
\[
\xi_0(s)=\int_0^\infty \frac{1-e^{-t}}{1+e^{-t}}\,se^{-ts}dt =s\mathcal{L} \left[ \frac{1-e^{-t}}{1+e^{-t}} \right]\!(s) =  \mathcal{L}\left[  \frac{2e^{-t}}{(1+e^{-t})^2} \right]\!(s), 
\]
where $\mathcal{L}$ is the Laplace transform. Similarly,
\[
1-\xi_0(s)= \int_0^\infty \frac{2e^{-t}}{1+e^{-t}}\,se^{-ts}dt=s \mathcal{L}\left[  \frac{2e^{-t}}{1+e^{-t}} \right]\!(s)
\]
and
\[
\xi'_0(s)=  - \mathcal{L}\left[  \frac{2te^{-t}}{(1+e^{-t})^2} \right]\!(s) .
\]


Thus $g$ takes the form
\begin{eqnarray*}
  g(s)&=& \mathcal{L}\left[  \frac{2e^{-t}}{(1+e^{-t})^2} \right]\!(s) \cdot s \mathcal{L}\left[  \frac{2e^{-t}}{1+e^{-t}} \right]\!(s)   -s \mathcal{L}\left[  \frac{2te^{-t}}{(1+e^{-t})^2} \right]\!(s) \\
     &=& 2s \mathcal{L} \left[  \frac{2e^{-t}}{(1+e^{-t})^2} * \frac{e^{-t}}{1+e^{-t}} -  \frac{te^{-t}}{(1+e^{-t})^2}   \right]\! (s).
\end{eqnarray*}
In order to calculate the convolution, we first find the primitive function:
\[
\int \frac{2e^{-x}}{(1+e^{-x})^2}\cdot \frac{e^{x-t}}{1+e^{x-t}}\,dx =
 \frac{2e^t \log(e^x+1)}{(e^t-1)^2} + \frac{2}{(e^t-1)(e^x+1)} - \frac{2e^t \log(e^t+e^x)}{(e^t-1)^2}+C.
\]
Thus
\begin{eqnarray*}
   \frac{2e^{-t}}{(1+e^{-t})^2} * \frac{e^{-t}}{1+e^{-t}}&=&  \frac{4e^t \log(e^t+1)}{(e^t-1)^2} -
   \frac{4e^t \log2}{(e^t-1)^2} - \frac{2te^t }{(e^t-1)^2} - \frac{1}{e^t+1} 
\end{eqnarray*}
and 
\begin{eqnarray*}
  \frac{g(s)}{2s} &=& \mathcal{L} \left[ \frac{4e^t \log(e^t+1)}{(e^t-1)^2} - \frac{4e^t \log2}{(e^t-1)^2} - \frac{2te^t }{(e^t-1)^2} - \frac{1}{e^t+1} -  \frac{te^{-t}}{(1+e^{-t})^2}    \right] (s)  \\
   &= & \mathcal{L} \left[  \frac{2e^t}{(e^t-1)^2}\left(2\log\frac{1+e^t}2-t\right) - \frac{1+e^t+te^t}{(e^t+1)^2}    \right] (s).
\end{eqnarray*}

To understand the behavior of this expression, consider functions $\psi_1$ and $\psi_2$ defined in Lemma~\ref{lem-uniq-root}. This leads us to the relation
\[
  \frac{g(s)}{2s} = \mathcal{L} \left[  \psi_1(e^t-1) - \psi_2(e^t-1)  \right] (s).
\]
Lemma~\ref{lem-uniq-root} states that the pre-image $\mathcal{L}^{-1}\!\!\left[  \frac{g(s)}{2s} \right]$ has a unique root for $t>0$. Then $\frac{g(s)}{2s} < 0$ by Corollary~\ref{corr-laplace}. So, inequality \eqref{g} holds, which completes the proof.
\end{proof}

It is worth mentioning that Theorem~\ref{lem-kappa}, in fact, presents certain properties of the values of the Gau{\ss} hypergeometric function at $z=-1$ because  functions $\xi_j$ can be expressed by it.

\begin{corollary}
Denote $F(s)= \Gauss(1,s;s+1;-1)$. The functions $ F(s)$ and $\frac{1- F(s)}{s}$ are decreasing while $s\!\left( F(s)-\frac12\right)$ and $\frac{1- F(s)}{s\left( 2F(s)-1\right)}$  are increasing on $(0,\infty)$. Moreover, the following sharp estimates hold:
\[
\frac12< F(s)<1,\  \    0<\frac{1- F(s)}{s}<\ln 2,\  \   0< s\!\left(\!\! F(s)-\frac12\!\right)\!\! <\frac14, \ \    \ln 2<  \frac{1- F(s)}{s\left( 2F(s)-1\right)}<1.
\] 
\end{corollary}

\bigskip

\section{A two-parameter family and inclusion property}\label{sect-2-param}

\setcounter{equation}{0}

 Denote by $\mathcal{A}$ the set of all holomorphic functions in the open unit disk $\D$ normalized by $f(0)=f'(0)-1=0$. Let $\Omega=\{(s,t):\, s\in[0,\infty),\,t\in[0,1)\}$.  
From now on we are dealing with the two-parameter family $\Ao$ consisting of the sets
\begin{equation}\label{2-param}
\Ao_s^t:=\left\{f\in\mathcal{A}:\ \Re \left[(s-1) \frac{f(z)}{z} + f'(z) \right]\geq st, \  z\in \D\setminus \{0\}\right\},  \quad (s,t)\in\overline\Omega,
\end{equation}
and $\Ao_\infty^t:= \left\{f\in\mathcal{A}:\ \Re \left[ \frac{f(z)}{z}  \right]\geq t, \  z\in \D\setminus \{0\}\right\}$.

These classes were introduced in~\cite{K-R}, where an integral transform between different sets $\Ao_s^t$ was established. The sets $\Ao_1^t$ were studied even earlier in~\cite{Hal}. 
Subsequently, in \cite{E-J-survey} we considered these classes with a different parametrization and found certain functions  $t=t(s)$ for which the sets $\Ao_s^{t(s)}$ form filtrations. 

\vspace{2mm}

The following facts are evident.
\begin{lemma}\label{lem-starting}
For each $(s,t)\in \overline{\Omega}$, the set $\Ao_s^t$ is a convex body. Moreover,
\begin{itemize}
 \item[(a)] $\Ao_0^t=\Ao_s^1=\{\Id\}$; \item[(b)] $f\in\Ao_\infty^t \iff \frac{f(z)-tz}{(1-t)z}\in\mathcal{C}$; \item[(c)] if $\ 0\le t_1<t_2\le1$, then $\Ao_s^{t_1}\supset \Ao_s^{t_2}$; \item[(d)] if $f(z)=zp(z)$, then $f\in\Ao_s^t\Leftrightarrow \Re \left[ s p(z) + zp'(z) \right]\ge st, \  z\in \D$.
     \end{itemize}
\end{lemma}

An additional useful property of the classes $\Ao_s^t$  was established in \cite{E-J-survey}: $${\inf_{f\in\Ao_s^t}\inf_{z\in\D}\Re\frac{f(z)}{z}=(1-t)\xi_0(s)+t.}$$

Since our primary focus of investigation  is the family $\Ao$ equipped with inclusion as the inherent partial order, this section is devoted to the subsequent relevant problem: 

\vspace{2mm}
$\bullet$ {\it Given two sets $\Ao_{s_1}^{t_1}$ and $\Ao_{s_2}^{t_2}$ of the family \eqref{2-param}, find conditions that entail or exclude the inclusion of one of them into the other.   }

Since the case $s_1=s_2$ is covered by assertion (c) of Lemma~\ref{lem-starting}, we advance, without loss of generality, assuming that $s_1<s_2$.

\begin{theorem}\label{thm-not_incl}
  Let $0\le s_1<s_2,\ t_1, t_2\in[0,1)$. Then $\Ao_{s_2}^{t_2} \not\subset \Ao_{s_1}^{t_1}$.
\end{theorem}

\begin{proof}
  By Lemma~\ref{lem-starting} (c), $\Ao_{s_1}^{t_1} \subset \Ao_{s_1}^0$. Hence, to prove our result, it suffices to find $f\in \Ao_{s_2}^{t_2}$ such that $f\not\in \Ao_{s_1}^0$ as $s_1<s_2$. 
  
  Let us define the function $p$ as follows
  \begin{eqnarray}\label{p}
  p(z) = 1+2(1-t)\left[\Gauss(1,s_2;s_2+1;z)-1\right] = 1+2(1-t)\sum_{n\ge1}\frac{s_2}{s_2+n}z^n   . 
  \end{eqnarray}
Formula \eqref{p} yields
\begin{equation}\label{p-prime}
p(z) + \frac{1}{s_2} zp'(z)= 1+2(1-t) \frac{z}{1-z}.
\end{equation}
Since the function $w=\frac{z}{1-z}$ maps the open unit disk $\D$ onto the half-plane $\Re w>-\frac12$, we conclude that $\inf _{z\in\D} \Re\left[p(z) + \frac1{s_2} zp'(z) \right] = t$. Thus the function $f$ defined by $f(z)=zp(z)$ belongs to $\Ao_{s_2}^{t_2}$  by Lemma~\ref{lem-starting}~(d). 

To show that $f\not\in \Ao_{s_1}^0$, let us consider the expression
\[
p(z) + \frac1{s_1} zp'(z) = \left(p(z) + \frac1{s_2} zp'(z)\right) +\left(\frac1{s_1}-\frac1{s_2} \right)  zp'(z) .
\]
We already know that the boundary values of $\Re\left(p(z) + \frac1{s_2} zp'(z)\right)$ equals $t$.
Since $s_1$ less than $s_2$ is arbitrary, it is enough to verify that the following claim holds:

{\bf Claim:} $\inf_{z\in\D} \Re \left[ zp'(z) \right]=-\infty.$\footnote{It seems that formula~(B18) in the book \cite{H-K-V} implies $\lim_{z\to1}\Re \left[ zp'(z) \right]=\infty$, which contradicts our claim. In this connection we notice  that the last formula is correct in the non-tangential sense only.}  

Indeed, function $p$ defined by \eqref{p} can be  represented by
\[
p(z)   =2t-1+2(1-t)\int_0^1 \frac{{s_2}x^{{s_2}-1}dx}{1-zx},
\]
see \eqref{gauss}. Combining this with \eqref{p-prime}, one concludes
\begin{eqnarray*}
 {s_2} zp'(z) &=& \left[1+2(1-t) \frac{z}{1-z}\right] -\left[ 2t-1+2(1-t)\int_0^1 \frac{{s_2}x^{{s_2}-1}dx}{1-zx} \right] \\
   &=& 2(1-t) \left[ 1+\frac{z}{1-z} - \int_0^1 \frac{{s_2}x^{{s_2}-1}dx}{1-zx} \right] \\
   &=& 2(1-t) \int_0^1 \left( \frac{1}{1-z} -\frac{1}{1-zx}   \right){s_2}x^{{s_2}-1}dx \\
   &=& 2(1-t) \int_0^1 \frac{z(1-x)}{(1-z)(1-zx)}  \,{s_2}x^{{s_2}-1}dx.
\end{eqnarray*}

Because the  hypergeometric function $\Gauss(1,{s_2};{s_2}+1;z)$ (and hence $p$) can be analytically extended at any boundary point $z\in\partial\D$ excepting $z=1$, we can put in the last formula $z=e^{i\phi},\ \phi\neq0$. In this case we get
\begin{eqnarray*}
 - \frac{{s_2}}{1-t} \Re \left. zp'(z)\right|_{z=e^{i\phi}} &=&- 2\Re \int_0^1 \frac{e^{i\phi}(1-x)}{(1-e^{i\phi})(1- e^{i\phi}x)}  \,{s_2}x^{{s_2}-1}dx  \\
   &=&- 2  \int_0^1 \Re \frac{(e^{i\phi}-1)(1-e^{-i\phi}x)(1-x)} {|1-e^{i\phi}|^2|1- e^{i\phi}x|^2}  \,{s_2}x^{{s_2}-1}dx \\
   &=&  \int_0^1 \frac{1-x^2}{|1-e^{i\phi}x|^2}  \,{s_2}x^{{s_2}-1}dx .
\end{eqnarray*}
Denote $\al_s:=\min\left\{ sx^{s-1}:\ x\in\left[ \frac12,1\right] \right\}$. Using this notation, we have
\begin{eqnarray*}
 - \frac{{s_2}}{1-t} \Re \left. zp'(z)\right|_{z=e^{i\phi}}
   \ge \int_{\frac12}^1 \frac{1-x^2}{|1-e^{i\phi}x|^2}  \,{s_2}x^{{s_2}-1}dx
  \ge \al_{s_2} \int_{\frac12}^1 \frac{1-x^2}  {1+x^2-2x\cos \phi}  \,dx  .
\end{eqnarray*}
Using the elementary calculus tools we get 
$$\int_{\frac12}^1 \frac{1-x^2}  {1+x^2-2x\cos \phi}  \,dx=-\cos \phi \cdot \ln (1-\cos \phi) + A(\phi),$$  where $A(\phi)$ is a bounded function.
Therefore this integral tends to infinity as $\phi\to0$. So, our Claim holds, which completes the proof.
\end{proof}

Thus, due to Theorem~\ref{thm-not_incl}, the inclusion $\Ao_{s_2}^{t_2} \subset \Ao_{s_1}^{t_1}$ is impossible when $s_1<s_2$. We present conditions ensuring the opposite inclusion that involve  function $\xi_0$ defined by~\eqref{xi_0}.

\begin{theorem}\label{lemma-g}
Let $(s_1,t_1)\in\Omega$ and  $s_1<s_2$. 
 \begin{itemize}
  \item [(i)]  If  $t_2=t_1+ (1-t_1)\left(1-\frac{s_1}{s_2}\right)\xi_0(s_1), $  then inclusion $\Ao_{s_1}^{t_1} \subset \Ao_{s_2}^{t_2}$ holds and is sharp in the sense that $\Ao_{s_1}^{t_1} \not\subset \Ao_{s_2}^t$ whenever~$t>t_2$. 
  \item [(ii)]   If  $\Ao_{s_1}^{t_1}\subset\Ao_{s_2}^{t_2}$,  then $t_2\leq t_1+ (1-t_1)\left(1-\frac{s_1}{s_2}\right)\xi_0(s_1).$ Consequently, $(1-t_2)s_2\ge(1-t_1) s_1$.
\item[(iii)]  In addition, if  $s_0 \in [0, s_1)$,  $\,t_0, t_2\in[0,1)$ and the inclusions  $\Ao_{s_0}^{t_0} \subset \Ao_{s_1}^{t_1}\subset \Ao_{s_2}^{t_2}$  hold, then the inclusion $\Ao_{s_0}^{t_0} \subset \Ao_{s_2}^{t_2}$  is not sharp.
\end{itemize}
\end{theorem}

\begin{proof}
By \eqref{2-param}, the identity mapping belongs to all classes $\Ao_s^t$.
Let $f\in\Ao_{s_1}^{t_1},\ f\neq\Id$. (So, $s_1\neq0$ by Lemma~\ref{lem-starting} (a).) This function can be represented in the form $f(z)=zp(z)$. It follows from Lemma~\ref{lem-starting} (d) that function $p$ satisfies the inequality 
\begin{equation}\label{in-filt}
\Re\left(s_1p(z)+zp'(z) \right)\ge s_1t_1.
\end{equation}
Therefore, the function $q$ defined by $q(z):=\frac{s_1p(z)+ zp'(z) - s_1t_1}{s_1(1-t_1) }$ satisfies $\Re q(z) \geq 0$ for all $z\in \D$ and $q(0)=1$. Then
\begin{equation*}
s_1p(z)+ zp'(z)=s_1\left(t_1 + (1-t_1)q(z)\right)=:q_1(z).
\end{equation*}
Function $p$ being the solution of this differential equation is
\begin{eqnarray}\label{fp-sol}%
p(z)=\int\limits_0^1 q_1\left(x z\right)x^{s_1-1} dx = t_1  + (1-t_1)\int\limits_0^1 q\left(x z\right) s_1 x^{s_1-1} dx.
\end{eqnarray}
By Harnack's inequality,
\begin{eqnarray*}
\Re p(z)\ge t_1+ (1 - t_1)\int\limits_0^1 \frac{1-x|z|} {1+x|z|}\,s_1 x^{s_1-1} dx .
\end{eqnarray*}
This inequality and \eqref{in-filt} imply
\begin{eqnarray*}
  \Re\left(s_2 p(z)+ zp'(z) \right) &=& \Re\left[ \left( s_2-s_1\right) p(z) + \left(s_1 p(z)+ zp'(z) \right)\right] \\
   &\ge& s_2 \left[ t_1+ (1 - t_1) \left(1-\frac{s_1}{s_2}\right)\int\limits_0^1 \frac{1-x|z|} {1+x|z|}\, s_1 x^{s_1-1} dx \right]\\
   &\ge&    s_2 \left[ t_1+ (1 - t_1) \left(1-\frac{s_1}{s_2}\right)\xi_0(s_1) \right] \!   ,  
\end{eqnarray*}
see \eqref{xi_0}. Thus $f\in\Ao_{s_2}^{t_2}$. To show that this estimate is sharp, let us choose function $q$ in \eqref{fp-sol} to be $q(z)=\frac{1-z}{1+z}$ and,
consequently,
\[
s_2 p(z)+ zp'(z) = s_2 t_1 +(1-t_1) \left[ s_1 \frac{1-z}{1+z} +\left(s_2 -s_1\right) \int_{0}^{1} \frac{1-xz}{1+xz}\,s_1x^{s_1-1} dx \right]\!. 
\]
Setting in this equality $z\to1^-$, we obtain statement (i).

Statement (ii) follows from (i) by direct calculations.

To prove (iii), we note that by statement (ii) the given inclusions imply  
\begin{equation}\label{aux-1}
\begin{array}{l}
 t_1 \le t_0+ (1-t_0)\left(1-\frac{s_0}{s_1}\right)\xi_0(s_0),\\
 t_2\le t_1+ (1-t_1)\left(1-\frac{s_1}{s_2}\right)\xi_0(s_1).
 \end{array}
\end{equation}

Assume by contradiction that the inclusion $\Ao_{s_0}^{t_0} \subset \Ao_{s_2}^{t_2}$  is sharp. Then $t_2$ is equal to $t_0+(1-t_0) \left(1-\frac{s_0}{s_2}\right)\xi_0(s_0)$ by statement (i). Comparing this fact with the second inequality in \eqref{aux-1} gives us
\[
t_0+(1-t_0) \left(1-\frac{s_0}{s_2}\right)\xi_0(s_0) \le t_1+ (1-t_1)\left(1-\frac{s_1}{s_2}\right)\xi_0(s_1).
\]
Note that the coefficient of $t_1$ in the right-hand side is positive. Therefore, one can replace $t_1$ by a larger expression. Taking in mind the first inequality in \eqref{aux-1} and reducing $(1-t_0)$, we~get 
\[
\frac{s_2-s_0}{s_2}\, \xi_0(s_0) \le \frac{s_1-s_0}{s_1}\, \xi_0(s_0) +  \left[1- \frac{s_1-s_0}{s_1} \xi_0(s_0)  \right] \! \cdot \frac{s_2-s_1}{s_2}\, \xi_0(s_1).
\]
This inequality is equivalent to
\begin{eqnarray*}
 \frac{(s_2-s_1)s_0}{s_1s_2}\, \xi_0(s_0)  &\le & \left[1- \frac{s_1-s_0}{s_1}\, \xi_0(s_0)  \right] \! \cdot \frac{s_2-s_1}{s_2}\, \xi_0(s_1), \\
    \frac{s_0}{s_1}\, \xi_0(s_0)  &\le & \left[1- \xi_0(s_0) + \frac{s_0}{s_1} \,\xi_0(s_0) \right] \! \cdot  \xi_0(s_1),  \\
      \frac{1}{s_1\xi_0(s_1)}  &\le &  \frac1{s_0\xi_0(s_0)} - \frac1{s_0} + \frac{1}{s_1}\, ,  \\
\end{eqnarray*}
which coincides with $\displaystyle \frac{1-\xi_0(s_1)}{s_1\xi_0(s_1)} \le \frac{1-\xi_0(s_0)}{s_0\xi_0(s_0)}. $ This contradicts statement (iv) of Theorem~\ref{lem-kappa}. The proof is complete.
\end{proof}

\bigskip

\section{Filtrations and quasi-extrema}\label{sect-filtra}
\setcounter{equation}{0}
 
In this section, we explore the set-theoretic structures  within the family of sets $\Ao_s^t$ defined by equation \eqref{2-param}. To do so, we introduce certain geometric objects tied to the outcomes of the preceding section.

Initially, let us recognize that the first statement (i) in Theorem~\ref{lemma-g} can be interpreted as follows. Given $P_0=(s_0,t_0)\in\Omega$, consider the function $t_{\uparrow, P_0}$ defined by 
\begin{equation}\label{fp-est1}
 t_{\uparrow, P_0}(s):=t_0+ (1-t_0)\left(1-\frac{s_0}{s}\right)\xi_0(s_0), \quad s\ge s_0.
 \end{equation}
 We designate its graph $\Gamma_{\uparrow,P_0}$ as the {\it forward extremal curve for the point $P_0$.}  Every point $P=(s,t)\in\Omega$ lying on or below this graph corresponds to the set $\Ao_s^t$ including $\Ao_{s_0}^{t_0}$, while all other points correspond to sets that do not include $\Ao_{s_0}^{t_0}$. 
In addition, if  $P_1\in \Gamma_{\uparrow,P_0}$, then  $\Gamma_{\uparrow,P_1}$  lies below $\Gamma_{\uparrow,P_0}$ by Theorem~\ref{lemma-g}~(iii).

Similarly, one can defined $\Gamma_{\downarrow,P_0}$, the {\it backward extremal curve for the point $P_0$.} This is the curve such that every point $P=(s,t)\in\Omega$ lying on or above it corresponds to the set $\Ao_s^t$ included in $\Ao_{s_0}^{t_0}$, while all other points correspond to sets not included in $\Ao_{s_0}^{t_0}$. $\Gamma_{\downarrow,P_0}$ is the graph of the implicit function $t_{\downarrow, P_0}$ defined by
\begin{equation}\label{fp-est-back}
 t_0=t_{\downarrow, P_0}(s)+ \left(1-t_{\downarrow, P_0}(s)\right) \left(1-\frac{s}{s_0}\right)\xi_0(s), 
 \end{equation}
 which is obviously well-defined and non-negative for all $s\in[s_*,s_0],$ where $s_*$ is the unique solution to the  equation $\left(1-\frac{s}{s_0}\right)\xi_0(s) =t_0 $. 
 
In this connection the following construction is natural and quite interesting. 
Start from a point $P_0=(s_0,t_0)\in\Omega$ and let $s_1=s_0+\Delta s$. If $\Delta s>0,$ calculate $t_1= t_{\uparrow,P_0}(s_1)$ (otherwise, we are dealing with  $t_{\downarrow,P_0}$). Continue by setting $s_2=s_1+\Delta s$ and $t_2= t_{\uparrow,P_1}(s_2)$. At the next step, let $s_3=s_2+\Delta s$, calculate $t_3$ by \eqref{fp-est1}, and so on. Letting $\Delta s\to0$, we obtain the differential equation $\displaystyle \frac{dt}{1-t}=\frac{\xi_0(s)ds}{s}$ with initial point $(s_0,t_0)$. Its solution is
\begin{equation}\label{fp-extr}
t_{P_0}(s)=1-(1-t_0)\exp\left[-\int_{s_0}^s \frac{\xi_0(\sigma)d\sigma}\sigma \right]\! .
\end{equation}
By construction, the graph $\Gamma_{P_0}$ of the last function has the peculiarity: if $P_1\in \Gamma_{P_0}$, then $\Gamma_{P_1}=\Gamma_{P_0}$. We say that this graph is the {\it curve of infinitesimally sharp inclusions}. 
The following result describes the relationship between the extremal curves and the curve of infinitesimally sharp inclusions.
\begin{theorem}\label{thm-compari}  
  Let $P_0\in\Omega$. Then   the curve of infinitesimally sharp inclusions  $\Gamma_{P_0}$  lies below  the forward  extremal curve $\Gamma_{\uparrow,P_0}$ and above the backward extremal curve $\Gamma_{\downarrow,P_0}$.
 \end{theorem}                  
 
\begin{proof}
To prove the first statement, compare the formulas \eqref{fp-est1} and \eqref{fp-extr}. We need to show that the inequality
 \[
1 - \exp\left[-\int_{s_0}^s \frac{\xi_0(\sigma)d\sigma}\sigma \right]<\left(1-\frac{s_0}{s}\right)\xi_0(s_0)
  \]
holds for all $s>s_0$. This is equivalent to $F(s)<0$, where
 \[
F(s):=\int_{s_0}^s \frac{\xi_0(\sigma)d\sigma}\sigma +\log\left(1-\xi_0(s_0)+\frac{1}{s}s_0\xi_0(s_0)\right)\!.
  \]
Assertion (iv) of Theorem~\ref{lem-kappa} implies 
\[
F'(s)=(\xi_3(s_0)-\xi_3(s))\cdot (s_0\xi_0(s_0)s\xi_0(s))<0.
\]
Since $F(s_0)=0$, this proves the desired.

Regarding the second assertion, we have $
t_{\downarrow, P_0}(s)=\frac{t_0 - \left(1-\frac{s}{s_0}\right)\xi_0(s)} {1 - \left(1-\frac{s}{s_0}\right)\xi_0(s)}$ by \eqref{fp-est-back}. So, the inequality $t_{\downarrow, P_0}(s)<t_{P_0}(s)$ for $s<s_0$ means that 
 $ \exp\left[-\int_{s_0}^s \frac{\xi_0(\sigma)d\sigma}\sigma \right] < \frac{1} {1 - \left(1-\frac{s}{s_0}\right)\xi_0(s)}$
  which is equivalent to $G(s)<0$, where 
\[
G(s):=- \int_{s_0}^s \frac{\xi_0(\sigma)d\sigma}\sigma + \log\left(1 - \left(1-\frac{s}{s_0}\right)\xi_0(s)\right).
  \]
  Since after the permutation $s_0 \leftrightarrow  s$, this function coincides with the function $F$ applied above, the proof is complete.
\end{proof}

\vspace{1mm}
 
We are at the point where we can address the main problems outlined in this paper.

\vspace{1mm}
Let $T:[s_*,\infty)\to[0,1)$ be a differentiable function.  The first inquiry is:

$\bullet$ {\it  What conditions on $T$ provide that the one-parameter family $\left\{\Ao_s^{T(s)},\ s\ge s*\right\}$ forms a filtration?}

We answer it as follows.

\begin{theorem}\label{thm-filtr}
  Let function $T$ be differentiable on $(s_*,\infty)$. Then $\Ao$ is a filtration if and only if 
  \begin{equation}\label{diff}
   T'(s) \le (1-T(s))\frac{\xi_0(s)}{s},\quad s>s_*.
  \end{equation}
  \end{theorem}
  
  \begin{proof}
   Let $s_0>s_*$ and analyze the function $F(s):=\log(1-T(s))-\log(1-t_{P_0}(s))$ with $P_0=(s_0,T(s_0))$. It follows from \eqref{fp-extr} that inequality \eqref{diff} means that $F'(s) \geq 0$.  Consequently, no part of the graph of $T$ can lie above the curve of infinitesimally sharp inclusions $\Gamma_{P_0}$.   

    Take any $s_1, s_2$  such that $s_*<s_1<s_2$. First assume that inequality \eqref{diff} holds. Then $T(s_2)\le t_{P_1}(s_2),\ P_1=(s_1,T(s_1))$. Therefore,  $\Ao_{s_1}^{T(s_1)}\subset \Ao_{s_2}^{T(s_2)}$ by Theorems~\ref{lemma-g} and~\ref{thm-compari}. Thus, since $s_1, s_2$ are arbitrary, we conclude that $\Ao$ is a filtration. 
    
    Otherwise, assume that $T'(s_1)> (1-T(s_1))\frac{\xi_0(s_1)}{s_1}$ for some $s_1>s_*$. Hence there is $s_2>s_1$ such that for all $s\in[s_1,s_2]$ the inequality $T'(s)> (1-T(s_1))\frac{\xi_0(s_1)}{s_2}$ holds. This implies 
    \[
    \frac{T(s_2)-T(s_1)}{s_2-s_1}> (1-T(s_1))\frac{\xi_0(s_1)}{s_2}\,,
    \]
    or, which is the same, $T(s_2)>T(s_1) + (1-T(s_1))\left( 1-\frac{s_1}{s_2}\right)\xi_0(s_1)=t_{\uparrow, P_1}(s_2)$. Hence $\Ao_{s_1}^{T(s_1)}\not\subset \Ao_{s_2}^{T(s_2)}$  by Theorem~\ref{lemma-g}, that is, $\Ao$ is not a filtration.   
\end{proof}

Now, we shift our attention to the {\it whole} family $\Ao$. 
As this family equipped with the relation $\subset$  constitutes a partially ordered family, our second inquiry is:

$\bullet$ {\it Does $(\Ao,\subset)$ indeed form a lattice?}

As we strive to comprehend this question, we uncover that the answer is negative, showing that the sets of so-called  quasi-suprema and quasi-infima are not singletons.

\begin{definition}\label{def-quasi}
  Given a pair $\Ao_1, \Ao_2\in\Ao$,  we say that 
\begin{itemize}
  \item $\Ao_0\in\Ao$ is a quasi-supremum of this pair and write $\Ao_0\in{\rm qsup}(\Ao_1,\Ao_2)$ if $\Ao_1\cup \Ao_2\subset\Ao_0$ and there is no $\Ao_*\in\Ao$ such that $\Ao_1\cup\Ao_2 \subset\Ao_*\subsetneq\Ao_0$.
  \item $\Ao_0\in\Ao$ is a quasi-infimum of this pair and write $\Ao_0\in{\rm qinf}(\Ao_1,\Ao_2)$ if $\Ao_0\subset\Ao_1\cap \Ao_2$ and there is no $\Ao_*\in\Ao$ such that $\Ao_0\subsetneq\Ao_*\subset\Ao_1\cap\Ao_2$.
\end{itemize}  
\end{definition}

We are now going to describe all quasi-suprema and quasi-infima of pairs of sets $\Ao_s^t$ defined by \eqref{2-param}.

Let $s_1\le s_2$ and the point  $(s_2,t_2)$ lies on or below the forward extremal curve $\Gamma_{\uparrow,P_1}$. Then $\Ao_{s_1}^{t_1} \subset \Ao_{s_2}^{t_2}$, and so $\Ao_{s_1}^{t_1}$ is the infimum as well as $\Ao_{s_2}^{t_2}$ is the supremum of this pair. Therefore we need to focus on the case $s_1<s_2$ and $t_2>t_1+ (1-t_1)\left(1-\frac{s_1}{s_2} \right) \xi_0(s_1)$.

\begin{theorem}\label{thm-supre}
Let $P_1=(s_1,t_1)\in\Omega$ and $P_2=(s_2,t_2)$ lie above $\Gamma_{\uparrow, P_1}$. Then the following assertions hold:
\begin{itemize}
  \item [(a)] the set ${\rm qsup}(\Ao_{s_1}^{t_1},\Ao_{s_2}^{t_2})$ 
consists of $\Ao_s^{\tau_1(s)}$ such that $s\ge s_2$ and  $  \tau_1(s)=\min\left\{ t_{\uparrow, P_1}(s), t_{\uparrow, P_2}(s) \right\}$;
  
  \item [(b)] the set ${\rm qinf}(\Ao_{s_1}^{t_1},\Ao_{s_2}^{t_2})$ 
consists of $\Ao_s^{\tau_2(s)}$ such that $s\le s_1$ and  $  \tau_2(s)=\max\left\{ t_{\downarrow, P_1}(s), t_{\downarrow, P_2}(s) \right\}$. 
\end{itemize}

\end{theorem}
\begin{proof}
We prove each one of the assertions by examining all points of $\Omega$. 

We commence with (a). If $s<s_2$ then $\Ao_{s_2}^{t_2}\not\subset\Ao_s^t$ according to Theorem~\ref{thm-not_incl}. 
If $s\ge s_2$ and $t>\tau_1(s)$, then by Theorem~\ref{lemma-g} either $\Ao_{s_1}^{t_1}\not\subset\Ao_s^t$ or $\Ao_{s_2}^{t_2}\not\subset\Ao_s^t$.  So,  $\Ao_s^t\not\in{\rm qsup}(\Ao_{s_1}^{t_1},\Ao_{s_2}^{t_2})$.

If $s\ge s_2$ and $t=\tau_1(s)$, then $\Ao_{s_1}^{t_1}\cup\Ao_{s_2}^{t_2} \subset\Ao_s^{\tau_1(s)}$  by Lemma~\ref{lem-starting} and Theorem~\ref{lemma-g}.
On the other hand, it follows from the above explanation that there is no  $\Ao_*\in\Ao$ such that $\Ao_{s_1}^{t_1}\cup\Ao_{s_2}^{t_2} \subset \Ao_*\subsetneq\Ao_s^t$. Thus $\Ao_s^t$ is a quasi-supremum.

If $s\ge s_2$ and $t<\tau_1(s)$, then $\Ao_{s_1}^{t_1}\cup\Ao_{s_2}^{t_2} \subset\Ao_s^{\tau_1(s)}\subsetneq\Ao_s^t$  by Lemma~\ref{lem-starting} and Theorem~\ref{lemma-g}.    Hence, $\Ao_s^t\not\in{\rm qsup}(\Ao_{s_1}^{t_1},\Ao_{s_2}^{t_2})$. Assertion (a) is proven.

Similarly to the above, if $s>s_1$ then $\Ao_s^t\not\subset\Ao_{s_1}^{t_1}$ according to Theorem~\ref{thm-not_incl}. If $s\le s_1$ and $t<\tau_2(s)$, then either $\Ao_s^t \not\subset\Ao_{s_1}^{t_1}$ or $\Ao_s^t \not\subset\Ao_{s_2}^{t_2}$  by Theorem~\ref{lemma-g}.    So,  $\Ao_s^t\not\in{\rm qinf}(\Ao_{s_1}^{t_1},\Ao_{s_2}^{t_2})$.

If $s\le s_1$ and $t=\tau_2(s)$, then $\Ao_s^{\tau_2(s)}\subset\Ao_{s_1}^{t_1}\cap\Ao_{s_2}^{t_2}$  by Lemma~\ref{lem-starting} and Theorem~\ref{lemma-g}.
In addition, there is no  $\Ao_*\in\Ao$ such that $\Ao_s^t\subsetneq\Ao_*\subset \Ao_{s_1}^{t_1}\cap\Ao_{s_2}^{t_2} $. Thus $\Ao_s^t$ is a quasi-infimum.

If $s\le s_1$ and $t>\tau_2(s)$, then $\Ao_s^t\subsetneq \Ao_s^{\tau_2(s)}\subset \Ao_{s_1}^{t_1} \cap \Ao_{s_2}^{t_2} $  by Lemma~\ref{lem-starting} and Theorem~\ref{lemma-g}.    Hence, $\Ao_s^t\not\in{\rm qinf}(\Ao_{s_1}^{t_1},\Ao_{s_2}^{t_2})$
\end{proof}

Observe that if a pair $\Ao_1, \Ao_2$ has the supremum, then by definition ${\rm sup}(\Ao_1, \Ao_2)\subset {\rm qsup}(\Ao_1, \Ao_2)$. On the other hand, Definition~\ref{def-quasi} implies that the relation ${\rm sup}(\Ao_1, \Ao_2)\subsetneq {\rm qsup}(\Ao_1, \Ao_2)$ is impossible. So, the quasi-supremum coincides with the supremum, in particular, it is unique.
Since not for all pairs $\Ao_{s_1}^{t_1},\Ao_{s_2}^{t_2}$ the sets ${\rm qsup}(\Ao_{s_1}^{t_1},\Ao_{s_2}^{t_2})$ and ${\rm qinf}(\Ao_{s_1}^{t_1},\Ao_{s_2}^{t_2})$ are singletons, we have:

\begin{corollary}
  The family $\Ao:= \left\{ \Ao_s^t:  (s,t)\in\overline\Omega\right\}$ is not a lattice.
\end{corollary}

\bigskip

\section{Upcoming questions }\label{sect-conclu}
\setcounter{equation}{0}

In the preceding sections, we introduced an approach for establishing set-theoretic properties of a family of sets consisting of holomorphic functions. We demonstrated the effectiveness of this method with a significant example involving sets  defined by \eqref{2-param}. Furthermore, it turns out that this approach relies on previously established characteristics of the hypergeometric function.  For this reason, it appears imperative that prior to effectively disseminating this approach, one should address the following question:

\begin{question}
  Expand Theorem~\ref{lem-kappa} to the case of $\Gauss(1,s;s+1;x)$, $x\in[-1,1]$, or a more general hypergeometric function $\Gauss(m,s;s+n;x)$ instead of $\Gauss(1,s;s+1;-1)$.
\end{question} 

An additional family that can be explored using the presented approach consists of the sets
\[
\mathfrak{B}_s^t:= \left\{f\in\mathcal{A}:\ \left| (s-1) \frac{f(z)}{z} + f'(z) -s \right| \le \frac{t}{1-t}, \  z\in \D\setminus \{0\}\right\},  \quad (s,t)\in\overline\Omega.
\]
These sets were studied in \cite{T-09} within the context of geometric function theory.  A recent investigation delved into the specific case where $\frac{t}{1-t}=1+s$, addressing problems in filtration theory in \cite{E-J-survey} and \cite{E-J-S}. We now pose the following questions:

\begin{question}
  What conditions on a function $T$ provide that the one-parameter family $\left\{\mathfrak{B}_s^{T(s)} \right\}$ forms a filtration?
\end{question} 

\begin{question}
  Is the family $\mathfrak{B}:=\left\{\mathfrak{B}_s^t,\ (s,t)\in\Omega\right\}$   a lattice?
  \end{question} 
  In the case of affirmative answer, the method of finding of the unique supremum and infimum for each pair of sets should be established. Otherwise, one asks about the sets of quasi-suprema and quasi-infima.

\vspace{2mm}
As for a general situation, we have already shown at the end of the previous section that if each pairs of elements of a family has the unique supremum (infimum), then the set of all quasi-suprema (quasi-infima) is a singleton. We do not know whether the converse statement is valid in general. At the same time, known examples lead us to the following

\begin{conjecture}
  A partially ordered family is a lattice if and only if each pair of its elements has a unique quasi-supremum and a unique quasi-minimum.
\end{conjecture}

\bigskip

\section*{Appendix}\label{sect-append}
\setcounter{equation}{0}

Here we prove Lemma~\ref{lem-uniq-root} that states that {\it the equation $\psi_1(x)=\psi_2(x)$, where
\[
\psi_1(x):=  \frac{2(1+x) }{ x^2}\,  \log\left(1+\frac{x^2}{4(1+x)}\right), \quad \psi_2(x):= \frac{2+x +(1+x)\log(1+x)}{(2+x)^2},
\]
has a unique solution in $(0,\infty)$.}

\begin{proof}
Our plan is the following: first we show that this equation has no solution for `small' $x$. Then we show that there is a unique solution for `large' $x$. In the last step we complete the proof. 

{\bf Step 1.} The inequality $$\zeta-\frac{\zeta^2}2+\frac{\zeta^3}3-\frac{\zeta^4}4<\log(1+\zeta) <\zeta-\frac{\zeta^2}2+\frac{\zeta^3}3,\quad \zeta>0,$$ implies
\begin{eqnarray*}
  \psi_1(x)  &<& \frac12   -\frac{x^2}{16(1+x)} +\frac{x^6}{6\cdot16(1+x)^2}  \\ 
  &=& \frac12 +\frac{x^2}{16}\left( -\frac1{1+x} +\frac{x^4}{6(1+x)^2}  \right)\!, \\
  \psi_2(x) &> & \frac1{2+x} +\frac{1+x}{(2+x)^2}\left( x-\frac{x^2}{2}+\frac{x^3}3-\frac{x^4}{4}\right)  \\
   &=& \frac12 +\frac{x^2}{16}\cdot\frac{1}{(2+x)^2}\left( -\frac{8x}{3} +\frac{4x^2}{3}-4x^3 \right)\! .
\end{eqnarray*}
Thus
\begin{eqnarray*}
\psi_2(x) -\psi_1(x)  > \frac{x^2}{96(1+x)^2}\cdot \phi(x),
\end{eqnarray*}
where  $\phi(x):=6+2x^2-x^4 -10x-24x^3 $. 
It can be easily seen that $\phi$ is a decreasing function that is positive at $x=0.4$. Hence $\psi_2(x)>\psi_1(x)$ in $(0,0.4]$.

{\bf Step 2.} Approximate computation gives us $\psi_1(10)<0.261<0.266<\psi_2(10)$. On the other hand, $\lim\limits_{x\to\infty}\frac{\psi_1(x)}{\psi_2(x)}=2$. Therefore, the equation has at least one solution in $[10,\infty)$.

Consider the equation $\frac{2+x}{\log(1+x)}\psi_1(x)=\frac{2+x}{\log(1+x)}\psi_2(x)$, which is equivalent to the given one. We state that the function in the left-hand side is increasing, while one in the right-hand side is decreasing. Indeed, it can be easily checked  that $\left( \frac{2+x}{\log(1+x)}\psi_2(x) \right)'<0$. 
The differentiation shows that the inequality $\left( \frac{2+x}{\log(1+x)}\psi_1(x) \right)'>0$ is equivalent~to
\[
\left[ 2x^2+2x+(3x+4)\log(1+x) \right] \log\frac{1+x}{1+\frac{x}{2}} >
\left[ 2x+(3x+4)\log(1+x) \right] \log \left(1+\frac{x}{2}\right)\!. 
\]
If $x>10$, then $\log\frac{1+x}{1+\frac{x}{2}}>0.606$. So, in this case it is enough to show that 
\[
1.212x^2>\left[  2x^2+2x+(3x+4)\log(1+x) \right] \log\frac{6+3x}{11}.
\]
 The last inequality follows from elementary calculus. Thus the equation $\psi_1(x)=\psi_2(x)$ has exactly one root in $x>10$.

{\bf Step 3.} To complete the proof, we have to show that there is no solution in $[0.4,10]$. Note that both $\psi_1$ and $\psi_2$ can be analytically  extended to the right half-plane. Hence one can find the number of solutions using the logarithmic residue of the function $\psi_1(z)-\psi_2(z)$ on the boundary of (for instance) the rectangle $\Omega=\left\{z=x+iy:\ 0.4\le x \le10,\ |y|\le2\right\}$. 

The approximate computation using Maple gives
\[
\frac1{2\pi i}\oint_{\partial D}\frac{\psi_1'(z)-\psi_2'(z)}{\psi_1(z)-\psi_2(z)}dz\approx -1\cdot 10^{-10} +0i.
\]
Since the logarithmic residue should be an integer, we conclude that it is zero, that is, there is no solution in $[0.4,10]$. The proof is complete.
\end{proof}

\bigskip

{\bf Acknowledgement}

\noindent The authors are grateful to Guy Katriel for very helpful discussions.

\bigskip

{\bf Declarations} 

\noindent {\bf Data availability} This manuscript has no associated data.

\noindent {\bf Conflict of interest} The authors declare that they have no Conflict of interest.

\noindent {\bf Ethical approval} Not applicable.

\noindent {\bf Financial interests} The authors have no relevant financial or non-financial interests to disclose.

\end{document}